\RequirePackage{fix-cm}
\documentclass[smallcondensed,envcountsame]{svjour3} 
\smartqed  
\usepackage{graphicx}
\usepackage{times,a4wide,mathrsfs}
\usepackage{amsmath}
\usepackage{amsfonts}

\newcommand{\C}{\mathbb{C}}
\newcommand{\ZZ}{\mathbb{Z}}

\newcommand{\QQ}{\mathbb{Q}}
\newcommand{\NN}{\mathbb{N}}
\newcommand{\PP}{\mathbb{P}}

\newcommand{\DD}{\mathcal D}
\newcommand{\XX}{\mathcal X}
\newcommand{\YY}{\mathcal Y}

\newcommand{\Z}{\mathcal Z}

\newcommand{\pic}{\hbox{Pic}}

\newcommand{\gr}{\hbox{Gr}}
\newcommand{\wt}{\widetilde}
\newcommand{\ima}{\hbox{Im}}
\newcommand{\rom}{\romannumeral}

\newtheorem{convention}{Conventions}

\newtheorem{nonumberingp}{Proposition}

 \journalname{}

\begin{document}

\title{Some results on a conjecture of Voisin for surfaces of geometric genus one}

\author{Robert Laterveer}

\institute{CNRS - IRMA, Universit\'e de Strasbourg \at
              7 rue Ren\'e Descartes \\
              67084 Strasbourg cedex\\
              France\\
              \email{laterv@math.unistra.fr}   }

\date{Received: date / Accepted: date}

\maketitle

\begin{abstract} Inspired by the Bloch--Beilinson conjectures, Voisin has formulated a conjecture concerning the Chow group of $0$--cycles on complex varieties of geometric genus one. This note presents some new examples of surfaces for which Voisin's conjecture is verified.

\end{abstract}



\section{Introduction}

The world of algebraic cycles on complex varieties is famous for its open questions (fairly comprehensive tourist guides, nicely exhibiting the boundaries between what is known and what is not known, can be found in \cite{Vo} and \cite{MNP}). The Bloch--Beilinson conjectures predict that this world has beautiful structure, and more precisely that there exists an intimate relation between Chow groups (i.e., algebraic cycles modulo rational equivalence) and singular cohomology.

The present note focuses on one particular instance of this predictive power of the Bloch--Beilinson conjectures: we consider the case of algebraic cycles on self--products $X\times X$, where $X$ is an $n$--dimensional smooth complex projective variety with $h^{n,0}=1$ and $h^{i,0}=0$ for all $0<i<n$. The Chow group of $0$--cycles
  \[ A^{2n}(X\times X) \]
  is conjecturally related to the cohomology groups
  \[  H^{4n}(X\times X), H^{4n-1}(X\times X), \ldots, H^{2n}(X\times X)\ .\]
Let
  \[ \iota\colon\ \ X\times X\ \to\ X\times X \]
  denote the involution exchanging the two factors. Then a consequence of this conjectural relation is that the effect of $\iota$ on $A^{2n}(X\times X)$ should be a reflection of the effect of $\iota$ on 
    \[  H^{4n}(X\times X), H^{4n-1}(X\times X), \ldots, H^{2n}(X\times X)\ .\]
 Now, the condition $h^{n,0}(X)=1$ ensures that the action of $\iota$ on $H^{2n}(X\times X)$ is particularly well--understood: we have that
   \[ (\hbox{id}  +(-1)^{n+1} \iota_\ast ) H^{2n}(X\times X)\ \subset\  H^{2n}(X\times X)\cap F^1\ ,\]
   where $F^\ast$ denotes the Hodge filtration
   (cf. lemma \ref{hodge} below).
   Conjecturally, this implies that
   \[   \hbox{id}= (-1)^n \iota_\ast\colon\ \ \gr^{2n}_F A^{2n}(X\times X)\ \to\  \gr^{2n}_F A^{2n}(X\times X)\ ,\]
   where $\gr^{2n}_F$ denotes the deepest level of the conjectural Bloch--Beilinson filtration on Chow groups. The condition on the Hodge numbers $h^{i,0}$ implies that all the levels $\gr^j_F$ for $j<2n$ are conjecturally $0$. Thus, one arrives at the following explicit conjecture concerning $0$--cycles on $X\times X$, which was first formulated by Voisin:

  \begin{conjecture}[Voisin \cite{V9}]\label{inv2} Let $X$ be a smooth projective complex variety of dimension $n$, with $h^{n,0}(X)=1$ and $h^{j,0}(X)=0$ for $0<j<n$. Let $z, z^\prime\in A^nX$ be $0$--cycles of degree $0$. Then
    \[   z\times z^\prime = (-1)^n \ z^\prime\times z\ \ \hbox{in} \ A^{2n}(X\times X)\ .\]
    (The notation $z\times z^\prime$ is a short--hand for the cycle class $(p_1)^\ast (z)\cdot (p_2)^\ast(z^\prime)\in A^{2n}(X\times X)$, where $p_1, p_2$ denote projection on the first, resp. second factor.)
    \end{conjecture}
    
   Loosely speaking: we have that almost all $0$--cycles are $(-1)^n \iota$--invariant.
   Conjecture \ref{inv2} is proven by Voisin for Kummer surfaces, and for a certain $10$--dimensional family of $K3$ surfaces \cite{V9}, obtained by desingularizing a double cover of $\PP^2$ branched along $2$ cubics.  
   
  The aim of this note is to add some more cases to the list of examples where conjecture \ref{inv2} is verified. The main ingredient we use is the theory of finite--dimensional motives of Kimura and O'Sullivan \cite{Kim}, \cite{An}, which did not exist at the time \cite{V9} was written.\footnote{Though reading with hindsight, it is clear that \cite{V9} already contains, {\sl avant la lettre\/}, many of the ideas of the theory of finite--dimensional motives -- in particular, the idea of considering the action of the symmetric group $S_k$ on the Chow groups of the product $X^k$.}

 \begin{nonumberingp}[(=propositions \ref{rho}, \ref{kunev}, \ref{SI}, \ref{reducible2} and \ref{6lines})] Let $X$ be one of the following: 
 
 \noindent
 {(\rom1)}
 a surface with $p_g=1$, $q=0$ which is $\rho$--maximal (in the sense of \cite{Beau2}) and has finite--dimensional motive (in the sense of \cite{Kim});
  
  \noindent
  {(\rom2)} a Kunev surface \cite{Tod};
  
  \noindent{(\rom3)}
  a $K3$ surface with a Shioda--Inose structure (for example, a $K3$ with Picard number $19$ or $20$);
  
  \noindent{(\rom4)}
  a $K3$ surface obtained from a double cover of $\PP^2$ branched along the union of an irreducible quadric and an irreducible quartic;
  
  \noindent{(\rom5)}
  a $K3$ surface obtained from a double cover of $\PP^2$ branched along $6$ lines.
  
  Then conjecture \ref{inv2} is true for $X$.
 \end{nonumberingp}
 
 Some explicit examples of families of surfaces of general type satisfying hypothesis (\rom1) are given in remark \ref{bonf}.
 A Kunev surface is a certain surface of general type with $q=0$ and $p_g=1$, these surfaces form a $12$--dimensional family \cite{Tod} (cf. definition \ref{kunev} for a precise definition). The generic member of a $K3$ surface as in (\rom4) has Picard number $9$. I am not aware of any $K3$ surface of Picard number less than $9$ for which conjecture \ref{inv2} is known, so obviously there is a lot of work remaining to be done ! 

\vskip0.6cm

\begin{convention} In this note, all varieties will be quasi--projective irreducible algebraic variety over $\C$, endowed with the Zariski topology. A {\sl subvariety\/} is a (possibly reducible) reduced subscheme which is equidimensional. 

{\bf All Chow groups will be with rational coefficients}: we will denote by $A_j(X)$ the Chow group of $j$--dimensional cycles on $X$ with $\QQ$--coefficients; for $X$ smooth of dimension $n$ the notations $A_j(X)$ and $A^{n-j}(X)$ will be used interchangeably. 

The notation $A^j_{hom}(X)$, resp. $A^j_{AJ}(X)$ will be used to indicate the subgroups of homologically trivial, resp. Abel--Jacobi trivial cycles.
For a morphism $f\colon X\to Y$, we will write $\Gamma_f\in A_\ast(X\times Y)$ for the graph of $f$.

In an effort to lighten notation, we will often write $H^j(X)$ or $H_j(X)$ to indicate singular cohomology $H^j(X,\QQ)$ resp. Borel--Moore homology $H_j(X,\QQ)$.

\end{convention}

\section{Finite--dimensional motives}

We refer to \cite{Kim}, \cite{An}, \cite{I}, \cite{MNP} for the definition of finite--dimensional motive. 
What mainly concerns us here is the nilpotence theorem, which embodies a crucial property of varieties with finite--dimensional motive:

\begin{theorem}[Kimura \cite{Kim}]\label{nilp} Let $X$ be a smooth projective variety of dimension $n$ with finite--dimensional motive. Let $\Gamma\in A^n(X\times X)_{}$ be a correspondence which is numerically trivial. Then there is $N\in\NN$ such that
     \[ \Gamma^{\circ N}=0\ \ \ \ \in A^n(X\times X)_{}\ .\]
\end{theorem}

 Actually, the nilpotence property (for powers of $X$) could serve as an alternative definition of finite--dimensional motive, as shown by a result of Jannsen \cite[Corollary 3.9]{J4}.

  Conjecturally, any variety has finite--dimensional motive \cite{Kim}. We are still far from knowing this, but at least there are quite a few non--trivial examples:
 
\begin{remark} 
The following varieties have finite--dimensional motive: varieties dominated by products of curves \cite{Kim}, $K3$ surfaces with Picard number $19$ or $20$ \cite{P}, surfaces not of general type with vanishing geometric genus \cite[Theorem 2.11]{GP}, Godeaux surfaces \cite{GP}, Catanese and Barlow surfaces \cite{V8}, many examples of surfaces of general type with $p_g=0$ \cite{PW}, Hilbert schemes of surfaces known to have finite--dimensional motive \cite{CM}, generalized Kummer varieties \cite[Remark 2.9(\rom2)]{Xu},
 3--folds with nef tangent bundle \cite{Iy} or \cite[Example 3.16]{V3}, 4--folds with nef tangent bundle \cite{Iy2}, log--homogeneous varieties in the sense of \cite{Br} (this follows from \cite[Theorem 4.4]{Iy2}), certain 3--folds of general type \cite[Section 8]{V5}, varieties of dimension $\le 3$ rationally dominated by products of curves \cite[Example 3.15]{V3}, varieties $X$ with Abel--Jacobi trivial Chow groups (i.e. $A^i_{AJ}(X)_{}=0$ for all $i$) \cite[Theorem 4]{V2}, products of varieties with finite--dimensional motive \cite{Kim}.
\end{remark}

\section{Surfaces that are $\rho$--maximal}
\label{secmain}

\begin{definition}[\cite{Beau2}] A smooth projective variety $X$ is said to be $\rho$--maximal if the rank $\rho$ of the Neron--Severi group is equal to the Hodge number $h^{1,1}$. 
\end{definition}

\begin{proposition}\label{rho} Let $X$ be a smooth projective variety of dimension $2$ with $p_g=1$ and $q=0$. Assume that $X$ has finite--dimensional motive, and that $X$ is $\rho$--maximal. Then for any $z,z^\prime\in A^2_{hom}(X)_{}$, one has
  \[  z\times z^\prime =z^\prime\times z\ \ \hbox{in}\ A^4(X\times X)_{}\ .\]
\end{proposition}

\begin{proof} 
Let $\iota$ denote the involution on $X\times X$ exchanging the two factors. The action of $\iota$ on cohomology is well--understood:

\begin{lemma}\label{hodge} Let $X$ be a surface with $q=0$ and $p_g=1$. We have
      \[  ( \Delta_{X\times X}- \Gamma_\iota )_\ast H^{4}(X\times X)\ \subset\  H^{4}(X\times X)\cap F^1\ \]
      (here $F^\ast$ denotes the Hodge filtration on $H^\ast(-,\C)$).
      \end{lemma}
     
     \begin{proof} The only summand in the K\"unneth decomposition of $H^{4}(X\times X)$ that is not in $F^1$ is $H^2X\otimes H^2X$. The correspondence
       \[ ( \Delta_{X\times X}- \Gamma_\iota )\]
       acts on
       \[  \ima\bigl( H^2X\otimes H^2X\ \to\ H^{4}(X\times X)\bigr) \]
       as twice the projector onto $\wedge^2 H^2X$. The lemma now follows from the following, which is \cite[Lemma 4.36]{Vo}.        
                
   \begin{lemma}\label{hodge2} Let $H$ be a Hodge structure of weight $n$ and with $\dim H^{n,0}=1$. Then the Hodge structure of weight $2n$ on 
   $\wedge^2 H$ has coniveau $\ge 1$.
   \end{lemma}
      \end{proof}        
 
 The $\rho$--maximality condition is used in the following guise:
  
 \begin{proposition}\label{support} Let $X$ be a $\rho$--maximal surface. Let $\alpha$ be a Hodge class
   \[ \alpha\in \Bigl((H^4(X\times X)\cap F^1) \otimes (H^4(X\times X)\cap F^1)\Bigr)\ \cap F^4  \ .\]
   Then there exists a divisor $D\subset X\times X$, and a cycle class $\gamma\in A_4(D\times D)$ such that
     \[  \gamma= \alpha\ \ \hbox{in}\ H^8(X^4)\ .\]
     \end{proposition}
     
  \begin{proof}
  Let 
  \[h^2=h^2_{alg}\oplus t_2(X)\]
 denote the decomposition of Chow motives as in \cite{KMP}, i.e. $t_2(X)=(X,\pi_2^{tr},0)$ is the transcendental motive of $X$ in the sense of loc. cit. Then the second cohomology group decomposes
  \[ H^2(X)= NS(X) \oplus W\ ,\]
  where $W=H^2(t_2(X))$.
 The $\rho$--maximality of $X$ implies that $W$ is a $2$--dimensional $\QQ$--vector space, since
 \[  W_{\C}=H^{0,2}(X)\oplus H^{2,0}(X)\ .\]
 
 We have that
   \[ \begin{split} H^4(X\times X)\cap F^1 &= (H^2(X)\otimes H^2(X))\cap F^1\\
                                                                     &= NS(X)\otimes NS(X) \oplus NS(X)\otimes W \oplus W\otimes NS(X) \oplus (W\otimes W)\cap F^1\\
                                                                       &= NS(X)\otimes NS(X) \oplus NS(X)\otimes W \oplus W\otimes NS(X) \oplus (W\otimes W)\cap F^2\ .\\      
                                                                 \end{split}\]    
  It is easy to prove the Hodge conjecture for $(W\otimes W)\cap F^2$:
  
  \begin{lemma}\label{WHC} The $\QQ$--vector space
     \[   (W\otimes W)\cap F^2 \subset H^4(X\times X)\cap F^2 \]
     is of dimension $1$, and generated by the cycle $\pi_2^{tr}\in A^2(X\times X)$.
 \end{lemma}
 
 \begin{proof} The complex vector space
   \[ F^2 (W_\C \otimes W_\C)= H^{0,2}(X)\otimes H^{2,0}(X) \oplus H^{2,0}(X)\otimes H^{0,2}(X)\]
   is $2$--dimensional, with generators $c,d$ such that $c=\bar{d}$. Let 
     \[ a\in (W\otimes W)\cap F^2\ ,\]
     i.e. $a$ is such that the complexification $a_\C\in H^4(X\times X,\C)$ can be written
   \[   a_\C= \lambda c +\mu \bar{c}\ .\]
   But the class $a_\C$, coming from rational cohomology, is invariant under conjugation, so that $\lambda=\mu$, i.e.
   \[ \dim (W\otimes W)\cap F^2 =1\ .\]
   The class of the cycle $\pi_2^{tr}$ in $H^4(X\times X)$ lies in $W\otimes W$ because 
    $W=H^2\bigl(t_2(X)\bigr)=(\pi_2^{tr})_\ast H^2(X)$.
 \end{proof}    
     
 By assumption, $\alpha$ is a Hodge class in
   \[  \begin{split}   &(H^4(X\times X)\cap F^1) \otimes (H^4(X\times X)\cap F^1)\\
                   &= \Bigl( NS(X)\otimes NS(X) \oplus \cdots \oplus (W\otimes W)\cap F^2\Bigr) \otimes      \Bigl( NS(X)\otimes NS(X) \oplus \cdots \oplus (W\otimes W)\cap F^2\Bigr)\ .\\
                   \end{split}\]
              It follows that $\alpha$ decomposes as a sum of Hodge classes $\alpha_1+\cdots+\alpha_{16}$ in the various components; we now analyze the various components that occur. 
              
              First, suppose there is a factor
              $NS(X)$ both in the first half and in the second half of the decomposition, e.g. consider
              \[  \alpha_6\in NS(X)\otimes W \otimes NS(X)\otimes W\ .\]
              This class $\alpha_6$ can be written
              \[ \alpha_6 = D_1 \times D_2\times \hat{\alpha}_6\ \ \in H^8(X^4)\ ,\]
              with $D_1, D_2\in NS(X)$ and $\hat{\alpha}_6\in W\otimes W$. Since $\alpha_6$ is a Hodge class, so is $\hat{\alpha}_6$. But then $\hat{\alpha}_6$ is algebraic, by lemma \ref{WHC}. It thus follows that $\alpha_6$ is represented by a cycle supported on divisor times divisor in $X^4$. 
              
              Next, suppose there is a factor $NS(X)$ on one side but not on the other side, e.g. consider
              \[ \alpha_8\in NS(X)\otimes W\otimes (W\otimes W)\cap F^2\ .\]
              Then the class $\alpha_8$ can be written as
              \[ \alpha_8=D\times \hat{\alpha_8}\times t(X)\ \ \in H^8(X^4)\ .\]
              Now $\hat{\alpha}_8$ is a Hodge class in $W$, so it must be $0$.
              The remaining cases are treated similarly.
              \end{proof}
 
 Proposition \ref{rho} is now easily proven: Let $\pi_2\in A^2(X\times X)$ denote a Chow--K\"unneth projector \cite{Mur}, \cite{KMP}. Using lemma \ref{hodge} and proposition \ref{support}, one obtains an equality between algebraic cycles modulo homological equivalence:
     \[  (\Delta_{X\times X}-\Gamma_\iota) \circ (\pi_2\times \pi_2)=\gamma\ \ \hbox{in}\ H^8(X^4) \ ,\]
     where $\gamma$ is a cycle supported on $D\times D$, for some divisor $D\subset X\times X$. This is equivalent to
     \[  (\pi_2\times\pi_2)-\Gamma_\iota\circ(\pi_2\times\pi_2)-\gamma=0\ \ \hbox{in}\ H^8(X^4)\ .\]
      Using the nilpotence theorem (theorem \ref{nilp}), this implies there exists  $N\in\NN$ such that
       \begin{equation}\label{Nth}  \Bigl( (\pi_2\times \pi_2)-\Gamma_\iota\circ (\pi_2\times \pi_2) -\gamma\Bigr)^{\circ N}=0\ \ \hbox{in}\ A^4(X^4) \ .\end{equation}
             
            Without loss of generality, we may suppose $N$ is odd.
            Define an integer 
        \[ M:= 1+\binom{N}{2} +\binom{N}{4}+\cdots +\binom{N}{N-1}=1+\binom{N}{N-2}+\binom{N}{N-4}+\cdots+\binom{N}{1}\ .\]    
             Upon developing (\ref{Nth}), we find an equality of correspondences
            \begin{equation}\label{eq} M \pi_2\times \pi_2- M \Gamma_\iota\circ (\pi_2\times \pi_2)=\sum_\ell Q_\ell\ \ \hbox{in}\ A^{4}(X^4)\ ,\end{equation}
where each $Q_\ell\in A^{4}(X^4)$ is a finite composition of correspondences
           \[ Q_\ell = Q_\ell^1\circ \ldots \circ Q_\ell^{N^\prime}\ \ \in A^{4}(X^4)\ \]
      for $N^\prime\le N$,  where $Q_\ell^j\in \{ (\pi_2\times \pi_2), \Gamma_\iota\circ(\pi_2\times \pi_2), \gamma\}$,  and at least one $Q_\ell^j$ is equal to $\gamma$. The correspondence $\gamma$ (being supported on $D\times D$ for some divisor $D$) does not act on $0$--cycles, so that
      \[ (Q_\ell)_\ast A^4(X\times X)=0\ \ \hbox{for\ all\ } Q_\ell\ .\]
    Applying equation (\ref{eq}) to $0$--cycles, we thus find that
      \[ \bigl( M (\pi_2\times \pi_2-\Gamma_\iota\circ(\pi_2\times \pi_2))\bigr)_\ast A^4(X\times X)=0\ ,\]
      i.e.
      \[  (\pi_2\times \pi_2)_\ast =(\Gamma_\iota\circ(\pi_2\times \pi_2))_\ast\colon\ \ A^4(X\times X)\ \to\ A^4(X\times X)\ .\]
     Since $\pi_2\times\pi_2$ acts as the identity on cycles of type $z\times z^\prime$ with $z,z^\prime\in A^2_{hom}X$, we have thus proven that
     \[ z\times z^\prime= z^\prime\times z\ \ \hbox{in}\ A^4(X\times X)\ ,\]
     i.e. conjecture \ref{inv2} is true for $X$.

    \end{proof}

\begin{remark} In particular, it follows from proposition \ref{rho} that a $K3$ surface with Picard number $20$ verifies conjecture \ref{inv2}; we will prove a more general result later (corollary \ref{1920}). For surfaces of general type with $p_g=K_X^2=1$, Beauville shows \cite[Proposition 9]{Beau2} that the $\rho$--maximal surfaces are dense in the moduli space. It would be interesting to prove that these surfaces have finite--dimensional motive.
\end{remark}

\begin{remark}\label{bonf} In \cite{Bonf}, Bonfanti constructs 2 families of surfaces of general type to which proposition \ref{rho} applies. These are the surfaces of type b and of type d in \cite[Table 1]{Bonf}, studied in detail in \cite[Sections 3.1 and 3.3]{Bonf}. All surfaces studied in \cite{Bonf} are dominated by products of curves and, as such, they have finite--dimensional motive. The $\rho$--maximality of the surfaces of type b and of type d is established in \cite[Section 4.1]{Bonf}.
\end{remark}

\section{Some special $K3$ surfaces}

\subsection{Double planes}

\begin{proposition}[Voisin \cite{V9}]\label{reducible} Let $X$ be a desingularization of the double cover of $\PP^2$ branched along the union of two irreducible cubics. Then conjecture \ref{inv2} is true for $X$.
\end{proposition}

\begin{proof} This is \cite[Theorem 3.4]{V9} (cf. also \cite[Section 4.3.5.2]{Vo}, \cite[Section 3]{V1}). Because we will use essentially the same argument in proposition \ref{reducible2} below, 
we briefly review Voisin's proof.
Let 
  \[f_1(x) ,\  f_2(x)\] 
  denote the equations of the two plane cubics, where $x=[x_0:x_1:x_2]\in\PP^2$. Let $\Sigma$ be the surface defined by
  \[ \Sigma= \bigl\{ [u:x_0:x_1:x_2]\in\PP^3\ \vert\    u^6=f_1(x)f_2(x)\bigr\}\ \ \subset \PP^3\ .\]
  There is a degree $3$ covering 
  \[ \begin{split} \psi\colon\ \ \Sigma\ &\to\ X\ ,\\  (u,x)\ &\mapsto\  (u^3,x)\\
       \end{split}\]
  (this corresponds to the quotient map $\PP^3\to \PP(1,1,1,3)$, since $X$ can be seen as the hypersurface in weighted projective space $\PP(1,1,1,3)$ given by $v^2=f_1(x)f_2(x)$).
  Let $W\subset\PP^5$ be the sextic fourfold defined by
  \[ f_1(x)f_2(x) - f_1(y)f_2(y)=0\ ,\]
  where $[x_0:x_1:x_2:y_0:y_1:y_2]$ are homogeneous coordinates for $\PP^5$.
  Let $\wt{W}\to W$ denote a desingularization. The fourfold $W$ is obviously invariant under the natural involution
    \[ \begin{split} i\colon\ \ \PP^5\ &\to\ \PP^5\ ,\\
        [x:y]\ &\mapsto\  [y:x]\ ;\\
        \end{split}\]
    likewise, $\wt{W}$ is $\wt{i}$--invariant, where $\wt{i}$ is the induced involution.
    
    There exists a (Shioda--style \cite{Shi}) rational map
    \[  \begin{split}\phi\colon\ \ \Sigma\times \Sigma\ &\dashrightarrow\ W\ ,\\
          \bigl( [u:x],[u^\prime:x^\prime]\bigr)\ &\mapsto\ [ u^\prime x: u x^\prime]\ ;\\
          \end{split} \]
          resolving indeterminacies one obtains a morphism 
     \[ \widetilde{\phi}\colon\ \ \widetilde{\Sigma\times \Sigma}\ \to\ \widetilde{W}\ .\]
  We now have defined morphisms
    \[  \begin{array}[c]{ccc}
                 \wt{\Sigma\times \Sigma}& \xrightarrow{\wt{\phi}}& \wt{W}\\
                   {\widetilde{\psi}\times\widetilde{\psi}}\ \ \ \downarrow\ \ \ \ \ \ \ \ \ \ \ \ \ \ \ \ &&\\
                   X\times X&&\\
                 \end{array}\]
        This induces a correspondence
        \[ \Gamma\in A^4(X\times X\times \widetilde{W})\ ,\]
        with action
        \[   \Gamma_\ast= \widetilde{\phi}_\ast (\wt{\psi}\times\wt{\psi})^\ast\colon\ \ A_i(X\times X)\ \to\ A_i(\wt{W})\ .\]       
   Analyzing the action of $\Gamma$, one directly checks that
     \[  \Gamma_\ast\colon\ \ \Bigl(  A_0^{hom}(X) \otimes A_0^{hom}(X) \Bigr)\ \to\ A_0(\wt{W}) \]
   is injective, and that 
      \[   \Gamma_\ast \Bigl ( a\times a^\prime - a^\prime\times a \Bigr)  \ \subset\ A_0(\wt{W})^{-} \ ,\]
    for any $a,a^\prime\in A_0^{hom}(X)$, where $A_0(\wt{W})^-$ denotes the $-1$--eigenspace for the action of $\wt{i}$ \cite[Lemma 3.4.1]{V9} (cf. also \cite[Lemma 3.5]{V1} for a slight variant, where a different involution on $\wt{W}$ is used).     
                      
  It remains to prove that the eigenspace $A_0(\wt{W})^-$ is $0$. To see this, one remarks that $W$ is covered by the family of (Calabi--Yau) $3$--folds $W_\alpha$, where for each $\alpha\in\C$, one defines
    \[  W_\alpha:=\bigl\{  [x:y]\in \PP^5\ \vert\ f_1(x)=\alpha f_2(y),\ f_1(y)=\alpha f_2(x)\bigr\}\ .\]
    Each $W_\alpha$ is $i$--invariant, and the general $W_\alpha$ is smooth. As each $0$--cycle on $W$ can be supported on finitely many smooth $W_\alpha$'s, the vanishing of the eigenspace $A_0(\wt{W})^-$ follows from the following result:
    
  \begin{proposition} Let $Z\subset\PP^5$ be a $3$--fold defined by two $i$--invariant cubic equations. Then $A_0(Z)^-=0$.
  \end{proposition}
  
 \begin{proof} This can be proven "by hand" using the method of \cite{V12}. 
 \end{proof}

\end{proof}

\begin{proposition}\label{reducible2} Let $X$ be a desingularization of the double cover of $\PP^2$ branched along the union of an irreducible quartic and an irreducible quadric. Then conjecture \ref{inv2} holds for $X$.
\end{proposition}

\begin{proof} This is similar to the above. Let 
  \[ f_1(x),\ f_2(x)\]
  be equations for the quartic resp. quadric in the branch locus, where $x=[x_0:x_1:x_2]$. 
  Let $W$ be the fourfold defined by
   \[   f_1(x)f_2(x) - f_1(y)f_2(y)=0\ .\]
 As $f_1f_2$ is of even degree, $W$ is invariant under the involution
   \[ \begin{split}  \tau\colon\ \ W\ &\to\ W\ ,\\ 
                     [x:y]\ &\mapsto\ [x:-y]\ .\\
                     \end{split}\]
  
  We let $\wt{W}\to W$ denote a resolution of singularities, and $\wt{\tau}$ the induced involution.   
 As above, there is a correspondence 
    \[ \Gamma \in A^4(X\times X\times \wt{W})\ ,\]
    inducing an injection
     \[  \Gamma_\ast\colon\ \ \Bigl(  A_0^{hom}(X) \otimes A_0^{hom}(X) \Bigr)\ \to\ A_0(\wt{W})\ . \]  
   We proceed to check that  
    \[   \Gamma_\ast \Bigl ( a\times a^\prime - a^\prime\times a \Bigr)  \ \subset\ A_0(\wt{W})^{-} \ ,\]
    for any $a,a^\prime\in A_0^{hom}(X)$, where now $A_0(\wt{W})^-$ denotes the $-1$--eigenspace for the action of $\wt{\tau}$.  To see this, note that Voisin \cite[Lemma 3.5]{V1} proves that
    \[  \Gamma_\ast      \Bigl ( a\times a^\prime - a^\prime\times a \Bigr)  \ \subset\ A_0(\wt{W})  \]
    is invariant under the involution $\wt{j}$ induced by
    \[ \begin{split} j\colon\ \ W\ &\to\ W\ ,\\
                                      [x:y]\ &\mapsto [y:-x]\\
                                    \end{split}\]
     (this involution $j$ is denoted $i$ in loc. cit.).   
     Note that we also have, as above in the proof of proposition \ref{reducible}, that
      \[  \Gamma_\ast      \Bigl ( a\times a^\prime - a^\prime\times a \Bigr)  \ \subset\ A_0(\wt{W})
        \]
      is anti--invariant under the involution $i$ exchanging $x$ and $y$. Since 
        \[  \tau=i\circ j\ ,\]
        it follows that 
        \[  \Gamma_\ast      \Bigl ( a\times a^\prime - a^\prime\times a \Bigr)  \ \subset\ A_0(\wt{W})  \]   
     is anti--invariant under $\wt{\tau}$, as claimed.
     
   It only remains to prove that $A_0(\wt{W})^-$, the anti--invariant part under $\wt{\tau}$, vanishes.
   To this end, we consider a family of (Calabi--Yau) $3$--folds $W_\alpha$ covering $W$, defined as
     \[   W_\alpha:=\bigl\{  [x:y]\in \PP^5\ \vert\ f_1(x)=\alpha f_1(y),\ f_2(y)=\alpha f_2(x)\bigr\}\ .\]
    Each $W_\alpha$ is $\tau$--invariant (since $f_1, f_2$ are of even degree), and the general $W_\alpha$ is smooth. As each $0$--cycle on $W$ can be supported on finitely many smooth $W_\alpha$'s, the vanishing of the eigenspace $A_0(\wt{W})^-$ now follows from the following result:
          
     \begin{proposition}\label{ok}  Let $Z\subset\PP^5$ be a smooth $3$--fold defined by two $\tau$--invariant equations of degree $2$ and $4$. Then $A_0(Z)^-=0$.
  \end{proposition}
  
 \begin{proof} Note that $Z$ is Calabi--Yau, and the involution $\tau$ acts as the identity on $H^{3,0}(Z)$, i.e. 
   \[ H^3(Z)^-\ \subset\  F^1 H^3(Z)\ .\]
 One invokes \cite[Proposition 2.1]{V12} to conclude that one has moreover
   \[  H^3(Z)^-\ \subset\ N^1 H^3(Z)\ ;\]
   what's more, $H^3(Z)^-$ is ``parametrized by algebraic cycles'' in the sense of \cite{V1}.
   Now one can apply the ``spreading out'' method of Voisin's \cite{V0}, \cite{V1} to the family of all smooth $\tau$--invariant complete intersections of multidegree $(2,4)$. Some care is needed because one does not have a complete linear system; this problem can be overcome as in \cite[Theorem 3.3]{V1}.
   
 Alternatively, one could prove proposition \ref{ok} ``by hand'' along the lines of \cite{V12}.  
  \end{proof}

 \end{proof}

\begin{proposition}\label{6lines} Let $X$ be a desingularization of the double cover of $\PP^2$ branched along $6$ lines in general position. Then conjecture \ref{inv2} is true for $X$.
\end{proposition}

\begin{proof} While this can probably be proven ``directly'' in the spirit of Voisin's result (proposition \ref{reducible}), we prefer to give a somewhat more ``fancy'' proof. This proof hinges on the fact that the Kuga--Satake construction for $X$ is algebraic \cite{Par}. More precisely, according to Paranjape \cite{Par} there exist an abelian variety $A$ of dimension $g$ and a correspondence $\Gamma^\prime\in A^2(X\times A\times A)$ such that
   \[   (\Gamma^\prime)_\ast\colon\ \ T_X\ \to\ H^2(A\times A)\]
   is an injection. It follows that there is an injection
   \[  \Gamma^\prime\colon t_2(X)\ \to\  h^2(A\times A)\ \ \hbox{in}\ {\mathcal M}_{num}\ ,\]
   where $t_2(X)$ is the transcendental motive of $X$ in the sense of \cite{KMP}, and ${\mathcal M}_{num}$ is the category of motives modulo numerical equivalence. Composing with some Lefschetz operator, one also gets an injection
   \[ \Gamma\colon\ \ t_2(X)\ \to\ h^{4g-2}(A\times A)\ \ \hbox{in}\ {\mathcal M}_{num}\ \]
   (here $\Gamma$ is the composition $L^{2g-2}\circ \Gamma^\prime$, where $L$ is an ample line bundle on $A\times A$).

   The category ${\mathcal M}_{num}$ being semi--simple \cite{J1}, this is a split injection, i.e. there exists a correspondence $\Psi\in A^{2}(A\times A\times X)$ such that
  \[   \Psi\circ \Gamma=\hbox{id}\colon\ \ t_2(X)\ \to\ t_2(X)\ \ \hbox{in}\ {\mathcal M}_{num}\ .\]
  But the motive $t_2(X)$ is finite--dimensional (it is a direct summand of $h(X)$, which is finite--dimensional since $X$ is dominated by a product of curves \cite{Par}). This implies that there exists $N\in\NN$ such that
    \[   \bigl( \Delta - \Psi\circ \Gamma \bigr)^{\circ N}=0\colon\ \ t_2(X)\ \to\ t_2(X)\ \ \hbox{in}\ {\mathcal M}_{rat}\ ,\]
    and hence that
    \[  \Gamma_\ast\colon\ \ A^2_{hom}(X)=A^2_{AJ}(X)= A^2_{}(t_2(X))\ \to\ A^{2g}_{AJ}(A\times A)\]
    is injective. 
    We note that, by construction, the action of $\Gamma$ on Chow groups factors as
    \[  \Gamma_\ast\colon\ \ A^2_{AJ}(X) \xrightarrow{ \Gamma^\prime} A^2(A\times A) \xrightarrow{L^{2g-2}} A^{2g}(A\times A)\ .\]
    Let $A^\ast_{(\ast)}()$ denote Beauville's filtration on Chow groups of abelian varieties \cite{Beau}. It follows that
    \[ \Gamma_\ast \bigl( A^2_{AJ}(X)\bigr)\ \subset\ \bigoplus_{j\le 2} A^{2g}_{(j)}(A\times A)\ ,\]
    as the Lefschetz operator preserves Beauville's filtration \cite{Kun}.
    On the other hand, 
    \[ \Gamma_\ast \bigl( A^2_{AJ}(X)\bigr)\ \subset\  A^{2g}_{AJ}(A\times A)= \bigoplus_{j\ge 2} A^{2g}_{(j)}(A\times A)\ .\]   
    The conclusion is that there is an injection
    \[ \Gamma_\ast\colon\ \ A^2_{AJ}(X)\ \to\ A^{2g}_{(2)}(A\times A)\ .\]

    The same argument gives also that
    \[  \Gamma\times\Gamma\colon\ \ \ima \Bigl( A^2_{hom}(X)\otimes A^2_{hom}(X)\ \to\ A^4(X\times X)\Bigr)\subset A^{4}(t_2(X)\otimes t_2(X))\ \to\ A^{4g}(A^4)\]
    is injective. It now suffices to prove a statement for the abelian variety $B=A\times A$:
    
    \begin{proposition}\label{abcod2} Let $B$ be an abelian variety of dimension $2g$. Let
      \[ a, a^\prime\ \ \in A^{2g}_{(2)}(B)\]
      be $2$ $0$--cycles. Then
      \[   a\times a^\prime - a^\prime\times a=0\ \ \hbox{in}\ A^{4g}(B\times B)\ .\]
      \end{proposition}
      
      \begin{proof} The group $A^{2g}_{(2)}(B)$ is generated by products of divisors
      \[   D_1\cdot D_2\cdot\ldots\cdot D_{2g}\ \ \in A^{2g}(B)\ ,\]
      with $2$ of the $D_j$ in $A^1_{(1)}(B)=\pic^0(B)$, and the remaining $2g-2$ $D_j$ in $A^1_{(0)}(B)$ \cite{Bl2}.
     As in \cite[Example 4.40]{Vo}, we consider the map
     \[ \sigma\colon B\times B\to B\times B,\ \ (a,b)\mapsto (a+b, a-b)\ .\]
     This is an isogeny, and one can check it induces a homothety on $A^\ast(B\times B)$. But on the other hand,
     \[  \sigma\circ \iota \circ \sigma= 2 (\hbox{id}_B, -\hbox{id}_B)\colon\ \ B\times B\ \to\ B\times B\ .\]
     It thus suffices to note that
     \[  (\hbox{id}_B, -\hbox{id}_B)_\ast \bigl(   D_1\cdot\ldots\cdot D_{2g}\times D_1^\prime\cdot\ldots\cdot D^\prime_{2g}\bigr)   = D_1\cdot\ldots\cdot D_{2g}\times D_1^\prime\cdot\ldots\cdot D^\prime_{2g}
         \ \  \hbox{in}\ A^{4g}(B\times B)\ ,\]
         since there is an even number of divisors $D_j^\prime$ for which $(-\hbox{id}_B)_\ast(D^\prime_j)=-D^\prime_j$ in $A^1B$.
         \end{proof}

         \end{proof}

\begin{remark} Note that the proof of proposition \ref{6lines} actually establishes something more general: if $X$ is a $K3$ surface with finite--dimensional motive, and the Kuga--Satake embedding of $X$ is induced by an algebraic cycle, then conjecture \ref{inv2} is true for $X$. For instance, this also applies to the quartic surface $X$ in $\PP^3$ defined by an equation
  \[  t^4=f(x,y,z)\ ,\]
where it is supposed that $f(x,y,z)=0$ defines a smooth quartic curve in $\PP^2$. (Indeed, the construction in \cite[Example 11.3]{vG} (where this example is attributed to Nori) shows that both hypotheses are fulfilled by $X$: the ``Kuga--Satake Hodge conjecture'' is shown to hold, and it is shown that $X$ is dominated by a product of curves so the motive is finite--dimensional.) Another example satisfying these conditions is \cite[Example 3.11]{vG2}, which is a $9$--dimensional family of elliptic $K3$ surfaces.
\end{remark}

\begin{remark} Improving on the results of this subsection, it would be interesting to consider more generally $K3$ surfaces that are double covers of $\PP^2$ ramified along an irreducible sextic. Voisin \cite{V1} proposes a tentative strategy towards settling conjecture \ref{inv2} for these $K3$ surfaces: applying \cite[Lemma 3.5]{V1} combined with (an improved variant of) \cite[Theorem 0.6]{V1}, it would suffice to prove that for a certain sextic fourfold $Y$ associated to $X$, one has that $F^1H^4(Y)$ is ``parametrized by algebraic cycles of dimension $1$'', in the sense of \cite{V1} (that is, it would suffice to prove a strong form of the generalized Hodge conjecture for $Y$).
 \end{remark}

\subsection{Shioda--Inose structure}

\begin{definition}[\cite{Mo}] For any surface $M$, let $T_M\subset H^2(M,\ZZ)$ denote the transcendental lattice. For $\ell\in\NN$, let $T_M(\ell)$ denote the lattice $T_M$ with intersection form multiplied by $\ell$. A {\em Nikulin involution\/} on a $K3$ surface $X$ is an involution acting as the identity on $H^{0,2}(X)$.

A $K3$ surface $X$ admits a {\em Shioda--Inose structure\/} if there exists a Nikulin involution $i$ on $X$ with rational quotient map 
  \[ \pi\colon\ \ X \dashrightarrow Y\ \]
  where $Y$ is a Kummer surface, and $\pi_\ast$ induces a Hodge isometry $T_X(2)\cong T_Y$.
\end{definition}

\begin{proposition}\label{SI} Let $X$ be a $K3$ surface with a Shioda--Inose structure. Then conjecture \ref{inv2} is true for $X$.
\end{proposition}

\begin{proof} As the Nikulin involution $i$ acts as the identity on $A^2X$ \cite{V11}, there is an isomorphism
  \[ \pi^\ast\colon\ \ A^2_{hom}(Y)\ \xrightarrow{\cong}\ A^2_{hom}(X)\ .\]
  The result now follows from the truth of conjecture \ref{inv2} for the Kummer surface $Y$ \cite{V9}.
  
    \end{proof}
    

\begin{corollary}\label{1920} Let $X$ be a $K3$ surface with Picard number $\ge 19$. Then conjecture \ref{inv2} is true for $X$.
\end{corollary}

\begin{proof} $X$ has a Shioda--Inose structure \cite[Corollary 6.4]{Mo}.
\end{proof}

\begin{remark} $K3$ surfaces admitting a Shioda--Inose structure are very special: their Picard number is at least $17$. For the case of Picard number $17$, explicit families of $K3$ surfaces with Shioda--Inose structure have been discovered: these are certain elliptic fibrations \cite{Ko}, \cite[4.7]{vGS}, as well as double covers of the plane branched along certain singular sextics \cite[4.5]{vGS}. More elliptic fibrations with a Shioda--Inose structure are given by \cite[Theorem 4.4]{CD}.

Note that a $K3$ surface admitting a Shioda--Inose structure and with Picard number $17$ or $18$ can not be a Kummer surface \cite[Corollary 3.7]{GS}.
\end{remark}

\begin{remark} It seems interesting to study conjecture \ref{inv2} in positive characteristic as well. As a starter, we note that corollary \ref{1920} still holds in positive characteristic, thanks to work of Liedtke \cite{Li}. More precisely, let $X$ be a $K3$ surface over an algebraically closed field of characteristic $\ge 5$. If the Picard number of $X$ is $22$, $X$ is unirational \cite[Theorem 5.3]{Li} so $A^2(X)$ is trivial. The Picard number can not be $21$ \cite[Theorem 2.6]{Li}. If the Picard number is $19$ or $20$, $X$ is dominated by a Kummer surface \cite[Theorem 2.6]{Li}, and the result follows since the result on abelian varieties \cite[Example 4.40]{Vo} still hold in positive characteristic.
\end{remark}

\subsection{Nikulin involutions}

There are many $K3$ surfaces $X$ with a Nikulin involution $i$ that is {\em not\/} a Shioda--Inose structure (e.g., when the quotient $K3$ surface is not a Kummer surface). Sometimes, we are lucky and the quotient $K3$ surface (more precisely, a minimal resolution of $X/i$) is one for which conjecture \ref{inv2} is known. In these cases, it follows that conjecture \ref{inv2} also holds for $X$. We give 2 examples of this phenomenon; one is a family of $K3$s with Picard number $9$, the other family has Picard number $16$.

\begin{proposition}\label{pic9} Let $X$ be a $K3$ surface such that the Neron--Severi group is isomorphic to the lattice $\Lambda_{\wt{4}}$, in the notation of \cite{vGS}. Then conjecture \ref{inv2} is true for $X$.
\end{proposition}

\begin{proof} The $11$--dimensional family ${\mathcal M}_{\wt{4}}$ of $K3$ surfaces of this type is described explicitly in \cite[3.5]{vGS}. In particular, it is shown in loc. cit. that there exists a Nikulin involution $i$ on $X$ such that a minimal resolution of the quotient $X/i$  is a $K3$ surface $Y$ isomorphic to a double plane with branch locus the union of a quartic and a conic. Conjecture \ref{inv2} is verified for such $Y$ (proposition \ref{reducible}). Since pull--back induces an isomorphism $A^2_{hom}(Y)\cong A^2_{hom}(X)$ \cite{V11}, it follows that conjecture \ref{inv2} holds for $X$.
\end{proof}
 
\begin{proposition}\label{pic16} Let $X$ be a generic $K3$ surface polarized by the lattice $H\oplus E_7\oplus E_7$, in the sense of \cite{CD}. Then conjecture \ref{inv2} is true for $X$.
\end{proposition}

\begin{proof} According to \cite[Theorem 4.4]{CD}, there is a Nikulin involution $i$ on $X$ such that a minimal resolution of the quotient $X/i$ is a $K3$ surface $Y$ isomorphic to a double cover of the plane branched along $6$ lines. Conjecture \ref{inv2} holds for $Y$ (proposition \ref{6lines}). Since pull--back induces an isomorphism $A^2_{hom}(Y)\cong A^2_{hom}(X)$ \cite{V11}, it follows that conjecture \ref{inv2} holds for $X$.
\end{proof}

\section{Kunev surfaces}

In this section we show that conjecture \ref{inv2} is true for Kunev surfaces. These surfaces form a $12$--dimensional family of surfaces of general type with $p_g=K_X^2=1$. The proof is quite direct, and goes as follows. The bicanonical map of a Kunev surface factors over a $K3$ surface, which is of a special type: it is obtained from a double cover of $\PP^2$ branched along the union of $2$ smooth cubics \cite{Tod}. By chance, for such $K3$ surfaces Voisin has already established the truth of conjecture \ref{inv2} (\cite{V9} or proposition \ref{reducible}). Hence, to prove conjecture \ref{inv2} for the Kunev surface $X$, it only remains to relate $0$-cycles on $X$ and $0$--cycles on the associated $K3$ surface; this can be done using the ``spreading out'' argument of \cite{V0} and \cite{V1}.

\begin{definition}[\cite{Tod}]\label{kunev} A {\em Kunev surface\/} is a smooth projective surface $X$ of general type with $p_g(X)=1$, $K_X^2=1$, such that its unique effective
canonical divisor is a smooth curve, and the morphism given by $\vert 2K_X\vert$ is a Galois covering of $\PP^2$.
\end{definition}    

\begin{remark} Surfaces of general type with $p_g=K_X^2=1$ are studied in \cite{Cat} and \cite{Tod}. In \cite{Cat}, a Kunev surface is called a {\em special\/} surface with $p_g=K^2_X=1$.
\end{remark}

 \begin{proposition}\label{propkunev} Let $X$ be a Kunev surface. Then conjecture \ref{inv2} is true for $X$.  
    \end{proposition}    
    
  \begin{proof} 
  
  According to the structural results of \cite{Tod} (or, independently, \cite{Cat}), any surface of general type with $p_g=K_X^2=1$ is a complete intersection of multidegree $(6,6)$ in a weighted projective space $P:=\PP(1,2,2,3,3)$.
  If in addition $X$ is a Kunev surface, then it is proven in \cite{Cat} and \cite{Tod} that the equations defining $X$ are invariant under the involution 
    \[\begin{split} i\colon\ \ &P\ \to\ P\ ,\\
                     & [x_0:x_1:\ldots:x_4]\ \mapsto\ [-x_0:x_1:\ldots:x_4]  \ .
                     \end{split}\]
           The quotient  $Y={X/ i}$ is a $K3$ surface, which is obtained by desingularizing a double cover of $\PP^2$ branched along two smooth cubics. Conjecture \ref{inv2} is true for $Y$ \cite[Theorem 3.4]{V9}. This implies conjecture \ref{inv2} for $X$, provided we can relate $0$--cycles on $X$ to $0$--cycles on $Y$; this is done in proposition \ref{same} below.
            \end{proof}
            
       \begin{proposition}\label{same} Let $X$ be a Kunev surface, and let $p\colon X\to Y$ denote the quotient map to the associated $K3$ surface. Then
       \[ p^\ast\colon\ \ A^2_{hom}(Y)\ \to\ A^2_{hom}(X) \]
       is an isomorphism.
        \end{proposition}     
        
        \begin{proof} We use the ``spreading out'' argument of Voisin's \cite{V0}, \cite{V1}, which exploits the fact that the surfaces come in a family. Let
        \[  \pi\colon\ \  \XX\ \to B\]
        denote the family of all smooth complete intersections in $P:=\PP(1,2,2,3,3)$, defined by $2$ equations of weighted degree $6$ where $x_0$ only occurs in even degree. For any $b\in B$, let $X_b$ denote the fibre $\pi^{-1}(b)$. The involution $i$ induces an involution on the total space of the family, which we still denote by $i$.
        This induces a quotient map
        \[  p\colon\ \ \XX\ \to\ \YY:=\XX/i\ ,\]
        where $\YY\to B$ is the family of associated $K3$ surfaces.
        
        Consider now the cycle
        \[  \DD:= \Delta- {1\over 2}{}^t\Gamma_p\circ \Gamma_p \ \ \in A^2(\XX\times_B \XX)\ \]
        (where $\Delta$ denotes the relative diagonal, and $\Gamma_p$ is the graph of $p$).
        This cycle has the property that for any $b\in B$, the restriction
        \[   \DD\vert_{X_b\times X_b}\ \ \in H^4(X_b\times X_b)\]
        is supported on $Z_b\times Z_b$, for some divisor $Z_b\subset X_b$. (Indeed, for any $b\in B$ we have that
          \[  (p_b)_\ast(p_b)^\ast(p_b)_\ast= 2(p_b)_\ast\colon\ \ H^{2,0}(X_b)\ \to\ H^{2,0}(Y_b)\ ,\]
          and hence
          \[  (p_b)^\ast(p_b)_\ast= 2\hbox{id}\colon\ \ H^{2,0}(X_b)\ \to\ H^{2,0}(X_b)\ .)\]          
         
         Using Voisin's ``spreading out'' result \cite[Proposition 2.7]{V0}, it follows there exists a divisor $\Z\subset\XX$ and a cycle $\DD^\prime\in A^2(\XX\times_B \XX)$ supported on $\Z\times_B \Z$, such that
         \[ (\DD-\DD^\prime)\vert_{X_b\times X_b}=0\ \ \hbox{in}\ \  H^4(X_b\times X_b)\ ,\]
         for all $b\in B$.
   Next, an analysis of the Leray spectral sequence as in \cite[Lemma 2.12]{V0} shows that there exists a cycle $\DD^{\prime\prime}$ with support on $\Z\times_B \XX \cup
   \XX\times_B \Z$, such that we have the global homological vanishing
      \[ \DD_{new}:=  \DD-\DD^\prime -\DD^{\prime\prime}=0\ \ \hbox{in}\ \ H^4(\XX\times_B \XX)\ \]
      (here we have enlarged the divisor $\Z\subset\XX$).
     Denoting by $f$ the blow--up of $\XX\times_B \XX$ along the relative diagonal, we also have
     \[  f^\ast(\DD_{new})=0\ \ \hbox{in}\ \ H^4(\widetilde{\XX\times_B \XX})\ .\] 
      
    Let $Q$ be the compactification of $\XX\times_B \XX$ introduced in lemma \ref{compact} below. The variety $Q$ is almost smooth: it is a quotient variety $Q=Q^\prime/G$, where $G$ is a finite group (because $P$ is a quotient variety). This implies there is a good intersection theory with rational coefficients on $Q$ \cite[Example 17.4.10]{F}. 
      Using the truth of the Hodge conjecture for divisors, we find there exists a cycle class
      \[    \overline{\DD}_{new}\ \ \in A^2_{hom}(Q)\]
      restricting to $f^\ast(\DD_{new})$.
      But the cycle $\overline{\DD}_{new}$ is rationally trivial (lemma \ref{compact}), hence so is its restriction to any fibre. This proves proposition \ref{same} for general $b\in B$: indeed, we find an equality
      \[   \Delta_{X_b}-{1\over 2}{}^t \Gamma_p \circ \Gamma_p=   (\DD^\prime+ \DD^{\prime\prime})\vert_{X_b\times X_b}\ \ \hbox{in}\ \ A^2(X_b\times X_b)\ ,\]
      and for general $b\in B$ the right--hand side does not act on $A^2_{hom}(X_b)=A^2_{AJ}(X_b)$.
      
      To get the result for any $b_0\in B$, it suffices to note that in the above construction, the divisor $\Z$ supporting the cycles $\DD^\prime$ and $\DD^{\prime\prime}$ may be chosen in general position with respect to $X_{b_0}$, and then the above argument applies to $X_{b_0}$.
      
      \begin{lemma}\label{compact} Set--up as above. Let
        \[ f\colon\ \ \widetilde{\XX\times_B \XX}\ \to\ \XX\times_B \XX\]
      be the blow--up along the relative diagonal, and let
      \[ \widetilde{P\times P}\ \to\ P\times P\]
      be the blow--up along the diagonal.  
        There exists a projective compactification
        \[   Q\ \supset\ \widetilde{\XX\times_B \XX} \ ,\]
        with the property that $Q$ is a fibre bundle over $\widetilde{P\times P}$, and fibres are products of projective spaces.
         In particular, we have
          \[ A^2_{hom}(Q)=0\ .\]
      \end{lemma}   
      
   \begin{proof} (This is inspired by Voisin's \cite[proof of proposition 2.13]{V0} (cf. also \cite[Lemma 1.3]{V1}, \cite[Lemma 4.32]{Vo}), which treats the slightly different case of the complete family of smooth complete intersections defined by very ample line bundles in an ambient space with trivial Chow groups.) 
 
 A point of $\wt{P\times P}$ is a triple $(x,y,z)$, where $x,y\in P$ and $z$ is a length $2$ subscheme of $P\times P$ with $z=x+y$. 
 Let $\bar{B}\supset B$ denote the product of projective spaces paremetrizing all pairs of (not necessarily smooth) weighted homogeneous polynomials of degree $6$ containing $x_0$ in even degree.
   The quasi--projective variety $\widetilde{\XX\times_B \XX}$ is contained in the projective variety $Q\subset \bar{B}\times \wt{P\times P}$ defined as
     \[  Q=\Bigl\{ \bigl((\sigma_1,\sigma_2), x,y,z\bigr)\in \bar{B}\times\wt{P\times P}\ \vert\ \sigma_1\vert_z=\sigma_2\vert_z=0 \Bigr\}\ \ \subset \ \bar{B}\times \wt{P\times P}\ .\]
    Let $p\colon Q\to \wt{P\times P}$ denote the projection. The fibre of $p$ over $(x,y,z)\in\wt{P\times P}$ is
      \[  p^{-1}(x,y,z)=\Bigl\{  (\sigma_1,\sigma_2)\in \bar{B}\ \vert\ \sigma_1\vert_z=\sigma_2\vert_z=0 \Bigr\}\ .\]
      
     We want to show that any fibre is a product of $2$ codimension $2$ linear subspaces in $\bar{B}$, i.e. that any $z$ imposes $2$ independent conditions on the polynomials $\sigma_j$. To this end, we note that there exists a degree $2$ map
     \[ \phi\colon\ \  P=\PP(1,2,2,3,3)\ \to\ \PP(2,2,2,3,3)=: P^\prime\ ,\] 
     and that the polynomials in $\bar{B}$ correspond to
     \[   \bar{B}^\prime:=\phi^\ast \vert {\mathcal O}_{P^\prime}(6)\vert \times  \phi^\ast \vert {\mathcal O}_{P^\prime}(6)\vert \ .\]
     It follows that the fibre $p^{-1}(x,y,z)$ is isomorphic to the subspace of $\bar{B}^\prime$ of polynomials passing through $\phi(z)$. But ${\mathcal O}_{P^\prime}(6)$ is a very ample line bundle on $P^\prime$ (this is proven in lemma \ref{delorme} below), so this subspace has codimension $2$.
     
     The conclusion about the vanishing of $A^2_{hom}(Q)$ follows from the fact that blow--ups and fibre bundle structures preserve the property of having trivial Chow groups \cite{V0}.    
     \end{proof}  
     
    \begin{lemma}\label{delorme} Let $P^\prime$ be the weighted projective space $\PP(2,2,2,3,3)$. Then the line bundle ${\mathcal O}_{P^\prime}(6)$ is very ample.
    \end{lemma}
    
    \begin{proof} The coherent sheaf ${\mathcal O}_{P^\prime}(6)$ is locally free, because $6$ is a multiple of the ``weights'' $2$ and $3$ \cite{Dol}. To see that this line bundle is very ample, we use the following numerical criterion: 
 
 \begin{proposition}[Delorme \cite{Del}]\label{del} Let $P=\PP(q_0, q_1,\ldots,q_n)$ be a weighted projective space. Let $m$ be the least common multiple of the $q_j$. Suppose every monomial
 \[x_0^{b_0} x_1^{b_1}\cdots x_n^{b_n}\]
 of (weighted) degree $km$ ($k\in \NN^\ast$) is divisible by a monomial of (weighted) degree $m$. Then ${\mathcal O}_{P}(m)$ is very ample.
 \end{proposition}
 
(This is the case $E(x)=0$ of \cite[Proposition 2.3(\rom3)]{Del}.)
 
 We apply proposition \ref{del} to the set--up of lemma \ref{delorme}. A monomial of degree $6k$ is of the form $x^{\underline{b}}=x_0^{b_0}\cdots x_4^{b_4}$ with
   \[ 2(b_0+b_1+b_2) +3(b_3+b_4)=6k\ .\]
   Suppose $b_3+b_4\ge 2$. Then the condition is obviously fulfilled, since we have a degree $6$ monomial $x_3x_4$ (or $x_3^2$ or $x_4^2$) dividing $x^{\underline{b}}$.
   So we may suppose $b_4=0$ and hence also $b_3=0$ (since $b_3=1$ would imply $6k$ is odd). Again, it is easily seen that the condition of the proposition is fulfilled: one can take an appropriate combination of $x_0, x_1, x_2$ to create a degree $6$ monomial dividing $x^{\underline{b}}$.
     \end{proof}   
      \end{proof}

  \begin{remark} There are two possible generalizations of proposition \ref{propkunev} that seem natural:
  
  The first is to try and extend proposition \ref{propkunev} to all surfaces of general type with $p_g=K_X^2=1$. Such surfaces are complete intersections in a weighted projective space \cite{Tod}, \cite{Cat}, so Voisin's method of spreading out cycles \cite{V0}, \cite{V1} applies. The ``only'' two obstacles that need to be circumvented are (1) that one needs the generalized Hodge conjecture for the Hodge structure $\wedge^2 H^2(X)\subset H^4(X\times X)$, and (2) that one needs the Voisin standard conjecture \cite[Conjecture 0.6]{V0} to get a cycle supported on some subvariety inside $X^4$.
  
 The other direction of generalization would be to extend proposition \ref{propkunev} to all {\em Todorov surfaces\/}, i.e. minimal surfaces $X$ of general type with $q=0$ and $p_g=1$
 having an involution $i$ such that $S/i$ is birational to a $K3$ surface and such that the bicanonical map of $X$ is composed with $i$. A Kunev surface is a Todorov surface with $K_X^2=1$. For any Todorov surface $X$, one can prove \cite{Rito} that the minimal resolution of $X/i$ is a $K3$ surface $Y$ obtained from a double plane with branch locus a union of $2$ cubics. As conjecture \ref{inv2} is known for such $Y$ (proposition \ref{reducible}), it ``only'' remains to show that $A^2_{hom}(X)\cong A^2_{hom}(Y)$. For the Kunev surfaces of proposition \ref{propkunev}, this was easy because they are complete intersections in a weighted projective space; for the other Todorov surfaces (i.e., with $K_X^2>1$), perhaps the total space of the family can likewise be exploited ? 
  \end{remark}  
   

\vskip1cm

\begin{acknowledgements} The ideas of this note grew into being during the Strasbourg 2014---2015 groupe de travail based on the monograph \cite{Vo}. Thanks to all the participants of this groupe de travail for a pleasant and stimulating atmosphere. Thanks to Olivier Benoist and Charles Vial for helpful conversations related to this note.
Many thanks to Yasuyo, Kai and Len for providing an environment propitious to work at home in Schiltigheim.
\end{acknowledgements}


\vskip1cm
\begin{thebibliography}{dlPG99}

\bibitem{An} Y. Andr\'e, Motifs de dimension finie (d'apr\`es S.-I. Kimura, P. O'Sullivan,...), S\'eminaire Bourbaki 2003/2004, Ast\'erisque 299 Exp. No. 929, viii, 115---145,




\bibitem{Beau} A. Beauville, Sur l'anneau de Chow d'une vari\'et\'e ab\'elienne, Math. Ann. 273 (1986), 647---651,

\bibitem{Beau2} A. Beauville, Some surfaces with maximal Picard number, Journal de l'Ecole Polytechnique Tome 1 (2014), 101---116,

\bibitem{Bl2} S. Bloch, Some elementary theorems about algebraic cycles on abelian varieties, Invent. Math. 37 (1976), 215---228,

\bibitem{B} S. Bloch, Lectures on algebraic cycles, Duke Univ. Press Durham 1980,


\bibitem{BS} S. Bloch and V. Srinivas, Remarks on correspondences and algebraic cycles, American Journal of Mathematics Vol. 105, No 5 (1983), 1235---1253,

\bibitem{Bonf} M. Bonfanti, On the cohomology of regular surfaces isogenous to a product of curves with $\chi({\mathcal O}_S)=2$, arXiv:1512.03168v1,

\bibitem{Br} M. Brion, Log homogeneous varieties, in: Actas del XVI Coloquio Latinoamericano de Algebra, 
Revista Matem\'atica Iberoamericana, Madrid 2007,
arXiv: math/0609669,


\bibitem{Cat} F. Catanese, Surfaces with $K^2=p_g=1$ and their period mapping, in:  Algebraic geometry (Copenhagen, 1978), Springer Lecture Notes in Mathematics, Springer 1979,

\bibitem{CD} Clingher and C. Doran, Note on a geometric isogeny of $K3$ surfaces, Int. Math. Research Notices 2011 (2011), 3657---3687,

\bibitem{CM} 
M. de Cataldo and L. Migliorini, The Chow groups and the motive of the Hilbert scheme of points on a
surface, Journal of Algebra 251 no. 2 (2002), 824---848,

\bibitem{Del} C. Delorme, Espaces projectifs anisotropes, Bull. Soc. Math. France 103 (1975), 203---223,



\bibitem{Dol} I. Dolgachev, Weighted projective varieties, in: Group actions and vector fields, Vancouver 1981, Springer Lecture Notes in Mathematics 956, Springer Berlin Heidelberg New York 1982,

\bibitem{F} W. Fulton, Intersection theory, Springer--Verlag Ergebnisse der Mathematik, Berlin Heidelberg New York Tokyo 1984,

\bibitem{GaP} A. Garbagnati and M. Penegini, $K3$ surfaces with a non--symplectic automorphism and product--quotient surfaces with cyclic groups, to appear in Rev. Mat. Iberoam.,

\bibitem{GS} A. Garbagnati and A. Sarti, Kummer surfaces and $K3$ surfaces with $(\ZZ/2\ZZ)^4$ symplectic action, arXiv:1305.3514,

\bibitem{vG} B. van Geemen, Kuga--Satake varieties and the Hodge conjecture, in:
The Arithmetic and Geometry of Algebraic Cycles, Banff 1998 (B. Gordon et alii, eds.), Kluwer Dordrecht 2000,

\bibitem{vG2} B. van Geemen, Half twists of Hodge structures of CM--type, J. Math. Soc. Japan Vol. 53 No. 4 (2001), 813---833,

\bibitem{vGS} B. van Geemen and A. Sarti, Nikulin involutions on $K3$ surfaces, Math. Z. 255 (2007), 731---753,


\bibitem{GP} V. Guletski\u{\i} and C. Pedrini, The Chow motive of the Godeaux surface, in:
Algebraic Geometry, a volume in memory of Paolo Francia (M.C. Beltrametti,
F. Catanese, C. Ciliberto, A. Lanteri and C. Pedrini, editors),
Walter de Gruyter, Berlin New York, 2002,

\bibitem{I} F. Ivorra, Finite dimensional motives and applications (following S.-I. Kimura, P. O'Sullivan and others), in:
        Autour des motifs, Asian-French summer school on algebraic geometry and number theory,
       Volume III, Panoramas et synth\`eses, Soci\'et\'e math\'ematique de France 2011,
          
\bibitem{Iy} J. Iyer, Murre's conjectures and explicit Chow--K\"unneth projectors for varieties with a nef tangent bundle, Transactions of the Amer. Math. Soc. 361 (2008), 1667---1681,

\bibitem{Iy2} J. Iyer, Absolute Chow--K\"unneth decomposition for rational homogeneous bundles and for log homogeneous varieties, Michigan Math. Journal
 Vol.60, 1 (2011), 79---91,
 

\bibitem{J1} U. Jannsen, 
Motives, numerical equivalence, and semi-simplicity, Invent. Math. 107(3) (1992), 447---452, 

\bibitem{J2} U. Jannsen, Motivic sheaves and filtrations on Chow groups, in: Motives (U. Jannsen et alii, eds.), Proceedings of Symposia in Pure Mathematics Vol. 55 (1994), Part 1,  


\bibitem{J4} U. Jannsen, On finite--dimensional motives and Murre's conjecture, in: Algebraic cycles and motives (J. Nagel and C. Peters, eds.), Cambridge University Press, Cambridge 2007,

\bibitem{KMP} B. Kahn, J. P. Murre and C. Pedrini, On the transcendental part of the motive of a surface, in: Algebraic cycles and motives (J. Nagel and C. Peters, eds.), Cambridge University Press, Cambridge 2007,

\bibitem{Kim} S. Kimura, Chow groups are finite dimensional, in some sense,
Math. Ann. 331 (2005), 173---201,



\bibitem{Ko} K. Koike, Elliptic $K3$ surfaces admitting a Shioda--Inose structure, Comment. Math. Univ. St. Pauli 61 No 1 (2012), 77---86,

\bibitem{Kun} K. K\"unnemann, A Lefschetz decomposition for Chow motives of abelian schemes, Inv. Math. 113 (1993), 85---102,





\bibitem{Li} C. Liedtke, Supersingular $K3$ surfaces are unirational, Invent. Math. 200 (2015), 979---1014,



\bibitem{Mo} D. Morrison, On $K3$ surfaces with large Picard number, Invent. Math. 75 No 1 (1984), 105---121,




\bibitem{Mur} J. Murre, On a conjectural filtration on the Chow groups of an algebraic variety, parts I and II, Indag. Math. 4 (1993), 177---201,

\bibitem{MNP} J. Murre, J. Nagel and C. Peters, Lectures on the theory of pure motives, Amer. Math. Soc. University Lecture Series 61, Providence 2013,




\bibitem{Par} K. Paranjape, Abelian varieties associated to certain K3 surfaces, Comp. Math. 68 (1988), 11---22,

\bibitem{P} C. Pedrini, On the finite dimensionality of a $K3$ surface, Manuscripta Mathematica 138 (2012), 59---72,


\bibitem{PW} C. Pedrini and C. Weibel, Some surfaces of general type for which Bloch's conjecture holds, to appear in: Period Domains, Algebraic Cycles, and Arithmetic, Cambridge Univ. Press, 2015,



\bibitem{Rito} C. Rito, A note on Todorov surfaces, Osaka Journal of Math. 46(3) (2009), 685---693,

\bibitem{R} A.A. Rojtman, The torsion of the group of 0--cycles modulo rational equivalence, Annals of Mathematics 111 (1980), 553---569,


\bibitem{Shi} T. Shioda, The Hodge conjecture for Fermat varieties, Math. Ann. 245 (1979), 175---184,


\bibitem{Tod} A. Todorov, Surfaces of general type with $p_g= 1$ and $(K,K) = 1$, Ann. Sci. de l'Ecole Normale Sup. 13 (1980), 1---21,


\bibitem{V} C. Vial, Algebraic cycles and fibrations, Documenta Math. 18 (2013), 1521---1553,

\bibitem{V2} C. Vial, Projectors on the intermediate algebraic Jacobians, New York J. Math. 19 (2013), 793---822,

\bibitem{V3} C. Vial, Remarks on motives of abelian type, to appear in Tohoku Math. J.,

\bibitem{V4} C. Vial, Niveau and coniveau filtrations on cohomology groups and Chow groups, Proceedings of the LMS 106(2) (2013), 410---444,

\bibitem{V5} C. Vial, Chow--K\"unneth decomposition for $3$-- and $4$--folds fibred by varieties with trivial Chow group of zero--cycles, J. Alg. Geom. 24 (2015), 51---80,



\bibitem{V9} C. Voisin, Remarks on zero--cycles of self--products of varieties, in: Moduli of vector bundles, Proceedings of the Taniguchi Congress  (M. Maruyama,  ed.), Marcel Dekker New York Basel Hong Kong 1994,


\bibitem{V12} C. Voisin, Sur les z\'ero--cycles de certaines hypersurfaces munies d'un automorphisme, Annali della Scuola Norm. Sup. di Pisa Vol. 29 (1993), 473---492,

\bibitem{V11} C. Voisin, Symplectic involutions of $K3$ surfaces act trivially on $CH_0$, Documenta Math. 17 (2012), 851---860,

\bibitem{V0} C. Voisin, The generalized Hodge and Bloch conjectures are equivalent for general complete intersections, Ann. Sci. Ecole Norm. Sup. 46, fascicule 3 (2013), 449---475,

\bibitem{V1} C. Voisin, The generalized Hodge and Bloch conjectures are equivalent for general complete intersections, II, J. Math. Sci. Univ. Tokyo  22 (2015), 491---517,

\bibitem{V8} C. Voisin, Bloch's conjecture for Catanese and Barlow surfaces, J. Differential Geometry 97 (2014), 149---175,

\bibitem{Vo} C. Voisin, Chow Rings, Decomposition of the Diagonal, and the Topology of Families, Princeton University Press, Princeton and Oxford, 2014,

\bibitem{Xu} Z. Xu, Algebraic cycles on a generalized Kummer variety, arXiv:1506.04297v1.

\end{thebibliography}
\end{document}